\documentclass[11pt]{amsart}
\usepackage{epsfig,amsfonts,amsthm,amssymb,latexsym,amsmath}

\theoremstyle{plain}
\newtheorem{theorem}{Theorem}[section]
\newtheorem{corollary}[theorem]{Corollary}
\newtheorem{lemma}[theorem]{Lemma}

\theoremstyle{remark}

\theoremstyle{definition}

\title[On Generalized Bihyperbolic Third-order Jacobsthal Polynomials]{On Generalized Bihyperbolic Third-order Jacobsthal Polynomials}

\vspace{5pt}

\author{\scriptsize GAMALIEL CERDA-MORALES}
\date{}

\begin{document}
\maketitle

\vspace{-20pt}
\begin{center}
{\footnotesize Instituto de Matem\'aticas,\\
Pontificia Universidad Cat\'olica de Valpara\'iso,\\
Blanco Viel 596, Valpara\'iso, Chile.\\
E-mail: gamaliel.cerda.m@mail.pucv.cl
}\end{center}

\vspace{20pt}
\hrule

\begin{abstract}
In this paper, a new generalization of third-order Jacobsthal bihyperbolic polynomials is introduced. Some of the properties of presented polynomials are given. A Vadja formula for the generalized bihyperbolic third-order Jacobsthal polynomials is obtained. This result implies the Catalan, Cassini and d'Ocagne identities. Moreover, generating function and matrix generators for these polynomials are presented.
\end{abstract}

\medskip
\noindent

\medskip
\noindent
\medskip

\hrule

\section{Introduction}
Let $\textbf{h}$ be the unipotent element such that $\textbf{h}\neq \pm 1$ and $\textbf{h}^{2}=1$. A hyperbolic number $\textrm{H}$ is defined as $\textrm{H}=a+b\textbf{h}$, where $a, b\in \mathbb{R}$. Denote by $\mathcal{H}$ the set of hyperbolic numbers. The hyperbolic numbers were introduced by Cockle (see \cite{Co1,Co2}).

The addition and subtraction of hyperbolic numbers is done by adding and subtracting the appropriate terms and thus their coefficients. The hyperbolic numbers multiplication can be made analogously as multiplication of algebraic expressions using the rule $\textbf{h}^{2}=1$. The real numbers $a$ and $b$ are called the real and unipotent parts of the hyperbolic number $\textrm{H}$, respectively. For others details concerning hyperbolic numbers see for example \cite{Po}.

In \cite{Ol}, Olariu introduced commutative hypercomplex numbers in different dimensions. One of $4$-dimensional commutative hypercomplex numbers is called the hyperbolic four complex number. In \cite{Po}, the authors used the name bihyperbolic numbers. Note that bihyperbolic numbers are a special case of generalized Segre's quaternions, being a $4$-dimensional commutative number system, and they are named as canonical hyperbolic quaternions (see \cite{Ca}). 

In this paper, we use the name bihyperbolic numbers. Analogously as bicomplex numbers are an extension of complex numbers, bihyperbolic numbers are a natural extension of hyperbolic numbers to $4$-dimension. 

Let $\mathcal{H}_{2}$ be the set of bihyperbolic numbers $\zeta$ of the form $$\zeta=a_{0}+a_{1}j_{1}+a_{2}j_{2}+a_{3}j_{3},$$ where $a_{i}\in \mathbb{R}$ ($i=0,1,2,3$), and $j_{1}, j_{2}, j_{3}\notin \mathbb{R}$ are operators such that
\begin{equation}\label{e1}
\begin{aligned}
j_{1}^{2}&=j_{2}^{2}=j_{3}^{2}=1,\\
j_{1}j_{2}&=j_{2}j_{1}=j_{3},\ j_{1}j_{3}=j_{3}j_{1}=j_{2},\ j_{2}j_{3}=j_{3}j_{2}=j_{1}.
\end{aligned}
\end{equation}

The addition and multiplication on $\mathcal{H}_{2}$ are commutative and associative. Moreover, $\left(\mathcal{H}_{2},+,\cdot\right)$ is a commutative ring. For the algebraic properties of bihyperbolic numbers review the work of Bilgin and Ersoy (see \cite{Bi}).

In this paper, we study some generalization of bihyperbolic third-order Jacobsthal numbers. The third-order Jacobsthal sequence $\{\mathcal{J}_{n}^{(3)}\}_{n\geq 0}$ is defined by the recurrence
\begin{equation}
\mathcal{J}_{n+3}^{(3)}=\mathcal{J}_{n+2}^{(3)}+\mathcal{J}_{n+1}^{(3)}+2\mathcal{J}_{n}^{(3)}\ \ (n\geq 0),
\end{equation}
with $\mathcal{J}_{0}^{(3)}=0$ and $\mathcal{J}_{1}^{(3)}=\mathcal{J}_{2}^{(3)}=1$. The Binet formula for the third-order Jacobsthal numbers has the form
\begin{equation}\label{b1}
\mathcal{J}_{n}^{(3)}=\frac{2^{n+1}}{7}-\frac{\omega_{1}^{n+1}}{(2-\omega_{1})(\omega_{1}-\omega_{2})}+\frac{\omega_{2}^{n+1}}{(2-\omega_{2})(\omega_{1}-\omega_{2})},
\end{equation}
where $\omega_{1}+\omega_{2}=-1$, and $\omega_{1}\omega_{2}=1$.

Some interesting properties of the third-order Jacobsthal numbers can be found in \cite{Co}. In the literature there are some generalizations of the third-order Jacobsthal numbers, see \cite{Ce1,Ce2,Ce3}. In \cite{Ce}, a one parameter generalization of the third-order Jacobsthal numbers was investigated. We recall this generalization.

For $x$ any variable quantity such that $x^{3}\neq 1$, the third-order Jacobsthal polynomials by the recurrence relation 
\begin{equation}\label{pol}
\mathcal{J}_{n+3}^{(3)}(x)=(x-1)\mathcal{J}_{n+2}^{(3)}(x)+(x-1)\mathcal{J}_{n+1}^{(3)}(x)+x\mathcal{J}_{n}^{(3)}(x),\ \ n\geq 0,
\end{equation}
with the initial conditions $\mathcal{J}_{0}^{(3)}(x)=0$, $\mathcal{J}_{1}^{(3)}(x)=1$ and $\mathcal{J}_{2}^{(3)}(x)=x-1$. It is easily seen that $\mathcal{J}_{n}^{(3)}(2)=\mathcal{J}_{n}^{(3)}$. Binet formula of third-order Jacobsthal polynomials $J_{n}^{(3)}(x)$ is given by 
\begin{equation}\label{bin1}
\mathcal{J}_{n}^{(3)}(x)=\frac{x^{n+1}}{x^{2}+x+1}-\frac{\omega_{1}^{n+1}}{(x-\omega_{1})(\omega_{1}-\omega_{2})}+\frac{\omega_{2}^{n+1}}{(x-\omega_{2})(\omega_{1}-\omega_{2})},
\end{equation}
where $x$ is any variable quantity such that $x^{3}\neq 1$, $\omega_{1}$ and $\omega_{2}$ are the roots of the characteristic equation $\lambda^{2}+\lambda+1=0$. 

In \cite{Da}, the Unrestricted Fibonacci quaternion $\mathcal{F}_{n}^{(a,b,c)}$ was introduced. For an integer $n$ and any integers $a$, $b$ and $c$, the generalized Fibonacci quaternion is defined by
$$\mathcal{F}_{n}^{(a,b,c)}=F_{n}+F_{n+a}i+F_{n+b}j+F_{n+c}k,$$
where $\{i,j,k\}$ is the standard basis of quaternions and $F_{n}$ is the $n$-th Fibonacci number. Motivated by the mentioned concept, in this paper, we introduce and study generalized bihyperbolic third-order Jacobsthal numbers.

\section{Generalized bihyperbolic third-order Jacobsthal polynomials}\label{sec:2}
Let $a\geq 1$, $b\geq 1$, $c\geq 1$ and $n\geq 0$ be integers, the $n$-th generalized bihyperbolic third-order Jacobsthal polynomial $\mathcal{BJ}_{n}^{(a,b,c)}(x)$ is defined as
\begin{equation}\label{e2}
\mathcal{BJ}_{n}^{(a,b,c)}(x)=\mathcal{J}_{n}^{(3)}(x)+\mathcal{J}_{n+a}^{(3)}(x)j_{1}+\mathcal{J}_{n+b}^{(3)}(x)j_{2}+\mathcal{J}_{n+c}^{(3)}(x)j_{3},
\end{equation}
where $\mathcal{J}_{n}^{(3)}(x)$ is the $n$-th third-order Jacobsthal polynomial and operators $j_{1}$, $j_{2}$, $j_{3}$ satisfy Eq. (\ref{e1}).

By Eq. (\ref{e2}) we obtain
\begin{equation}\label{e3}
\begin{aligned}
\mathcal{BJ}_{0}^{(a,b,c)}(x)&=\mathcal{J}_{0}^{(3)}(x)+\mathcal{J}_{a}^{(3)}(x)j_{1}+\mathcal{J}_{b}^{(3)}(x)j_{2}+\mathcal{J}_{c}^{(3)}(x)j_{3},\\
\mathcal{BJ}_{1}^{(a,b,c)}(x)&=\mathcal{J}_{1}^{(3)}(x)+\mathcal{J}_{a+1}^{(3)}(x)j_{1}+\mathcal{J}_{b+1}^{(3)}(x)j_{2}+\mathcal{J}_{c+1}^{(3)}(x)j_{3},\\
\mathcal{BJ}_{2}^{(a,b,c)}(x)&=\mathcal{J}_{2}^{(3)}(x)+\mathcal{J}_{a+2}^{(3)}(x)j_{1}+\mathcal{J}_{b+2}^{(3)}(x)j_{2}+\mathcal{J}_{c+2}^{(3)}(x)j_{3}.
\end{aligned}
\end{equation}
For $a=1$, $b=2$, $c=3$, we obtain the definition of the $n$-th bihyperbolic third-order Jacobsthal polynomial $\mathcal{BJ}_{n}^{(1,2,3)}(x)$, i.e., $\mathcal{BJ}_{n}^{(1,2,3)}(x)=\mathcal{BJ}_{n}(x)$.

By the definition of the generalized bihyperbolic third-order Jacobsthal polynomial we get the following recurrence relations.
\begin{theorem}
Let $n\geq 3$, $a\geq1$, $b\geq1$ and $c\geq1$ integers. Then,
\begin{equation}\label{rec}
\mathcal{BJ}_{n}^{(a,b,c)}(x)=(x-1)\mathcal{BJ}_{n-1}^{(a,b,c)}(x)+(x-1)\mathcal{BJ}_{n-2}^{(a,b,c)}(x)+x\mathcal{BJ}_{n-3}^{(a,b,c)}(x),
\end{equation}
where $\mathcal{BJ}_{0}^{(a,b,c)}(x)$, $\mathcal{BJ}_{1}^{(a,b,c)}(x)$ and $\mathcal{BJ}_{2}^{(a,b,c)}(x)$ are given by Eq. (\ref{e3}).
\end{theorem}
\begin{proof}
Using Eqs. (\ref{e2}) and (\ref{pol}), we have
\begin{align*}
(x-1)&\mathcal{BJ}_{n-1}^{(a,b,c)}(x)+(x-1)\mathcal{BJ}_{n-2}^{(a,b,c)}(x)+x\mathcal{BJ}_{n-3}^{(a,b,c)}(x)\\
&=(x-1)\left(\mathcal{J}_{n-1}^{(3)}(x)+\mathcal{J}_{n+a-1}^{(3)}(x)j_{1}+\mathcal{J}_{n+b-1}^{(3)}(x)j_{2}+\mathcal{J}_{n+c-1}^{(3)}(x)j_{3}\right)\\
&\ \ + (x-1)\left(\mathcal{J}_{n-2}^{(3)}(x)+\mathcal{J}_{n+a-2}^{(3)}(x)j_{1}+\mathcal{J}_{n+b-2}^{(3)}(x)j_{2}+\mathcal{J}_{n+c-2}^{(3)}(x)j_{3}\right)\\
&\ \ + x\left(\mathcal{J}_{n-3}^{(3)}(x)+\mathcal{J}_{n+a-3}^{(3)}(x)j_{1}+\mathcal{J}_{n+b-3}^{(3)}(x)j_{2}+\mathcal{J}_{n+c-3}^{(3)}(x)j_{3}\right)\\
&=\mathcal{J}_{n}^{(3)}(x)+\mathcal{J}_{n+a}^{(3)}(x)j_{1}+\mathcal{J}_{n+b}^{(3)}(x)j_{2}+\mathcal{J}_{n+c}^{(3)}(x)j_{3}\\
&=\mathcal{BJ}_{n}^{(a,b,c)}(x).
\end{align*}
The proof is completed.
\end{proof}

In the proof of the next theorem we will use the following result (see \cite{Ce}).
\begin{lemma}\label{lem1}
For $x$ any variable quantity such that $x^{3}\neq 1$, we have
\begin{equation}
\mathcal{J}_{n+2}^{(3)}(x)=-\mathcal{J}_{n+1}^{(3)}(x)-\mathcal{J}_{n}^{(3)}(x)+x^{n+1}.
\end{equation}
\end{lemma}
\begin{theorem}
Let $n\geq 0$, $a\geq1$, $b\geq1$, $c\geq1$ be integers. Then,
$$\mathcal{BJ}_{n+2}^{(a,b,c)}(x)+\mathcal{BJ}_{n+1}^{(a,b,c)}(x)+\mathcal{BJ}_{n}^{(a,b,c)}(x)=x^{n+1}\left(1+x^{a}j_{1}+x^{b}j_{2}+x^{c}j_{3}\right).$$
\end{theorem}
\begin{proof}
By the equality (\ref{e2}) and Lemma \ref{lem1}, we have
\begin{align*}
\mathcal{BJ}_{n+2}^{(a,b,c)}(x)&+\mathcal{BJ}_{n+1}^{(a,b,c)}(x)+\mathcal{BJ}_{n}^{(a,b,c)}(x)\\
&=\mathcal{J}_{n+2}^{(3)}(x)+\mathcal{J}_{n+a+2}^{(3)}(x)j_{1}+\mathcal{J}_{n+b+2}^{(3)}(x)j_{2}+\mathcal{J}_{n+c+2}^{(3)}(x)j_{3}\\
&\ \ + \mathcal{J}_{n+1}^{(3)}(x)+\mathcal{J}_{n+a+1}^{(3)}(x)j_{1}+\mathcal{J}_{n+b+1}^{(3)}(x)j_{2}+\mathcal{J}_{n+c+1}^{(3)}(x)j_{3}\\
&\ \ + \mathcal{J}_{n}^{(3)}(x)+\mathcal{J}_{n+a}^{(3)}(x)j_{1}+\mathcal{J}_{n+b}^{(3)}(x)j_{2}+\mathcal{J}_{n+c}^{(3)}(x)j_{3}\\
&=\mathcal{J}_{n+2}^{(3)}(x)+\mathcal{J}_{n+1}^{(3)}(x)+\mathcal{J}_{n}^{(3)}(x)\\
&\ \ + \left(\mathcal{J}_{n+a+2}^{(3)}(x)+\mathcal{J}_{n+a+1}^{(3)}(x)+\mathcal{J}_{n+a}^{(3)}(x)\right)j_{1}\\
&\ \ + \left(\mathcal{J}_{n+b+2}^{(3)}(x)+\mathcal{J}_{n+b+1}^{(3)}(x)+\mathcal{J}_{n+b}^{(3)}(x)\right)j_{2}\\
&\ \ + \left(\mathcal{J}_{n+c+2}^{(3)}(x)+\mathcal{J}_{n+c+1}^{(3)}(x)+\mathcal{J}_{n+c}^{(3)}(x)\right)j_{1}\\
&=x^{n+1}\left(1+x^{a}j_{1}+x^{b}j_{2}+x^{c}j_{3}\right).
\end{align*}
\end{proof}

Now, we give the Binet formula for the generalized bihyperbolic third-order Jacobsthal polynomials.
\begin{theorem}[Binet formula for the sequence $\mathcal{BJ}_{n}^{(a,b,c)}(x)$]\label{bb}
Let $n\geq 0$, $a\geq1$, $b\geq1$, $c\geq1$ be integers. Then,
$$
\mathcal{BJ}_{n}^{(a,b,c)}(x)=\frac{x^{n+1}\Theta}{x^{2}+x+1}-\frac{\omega_{1}^{n+1}\Phi_{1}}{(x-\omega_{1})(\omega_{1}-\omega_{2})}+\frac{\omega_{2}^{n+1}\Phi_{2}}{(x-\omega_{2})(\omega_{1}-\omega_{2})},
$$
where $\Theta=1+x^{a}j_{1}+x^{b}j_{2}+x^{c}j_{3}$, $\Phi_{1}=1+\omega_{1}^{a}j_{1}+\omega_{1}^{b}j_{2}+\omega_{1}^{c}j_{3}$ and $\Phi_{2}=1+\omega_{2}^{a}j_{1}+\omega_{2}^{b}j_{2}+\omega_{2}^{c}j_{3}$.
\end{theorem}
\begin{proof}
By the formulas (\ref{e2}) and (\ref{bin1}) we get
\begin{align*}
\mathcal{BJ}_{n}^{(a,b,c)}(x)&=\mathcal{J}_{n}^{(3)}(x)+\mathcal{J}_{n+a}^{(3)}(x)j_{1}+\mathcal{J}_{n+b}^{(3)}(x)j_{2}+\mathcal{J}_{n+c}^{(3)}(x)j_{3}\\
&=\frac{x^{n+1}}{x^{2}+x+1}-\frac{\omega_{1}^{n+1}}{(x-\omega_{1})(\omega_{1}-\omega_{2})}+\frac{\omega_{2}^{n+1}}{(x-\omega_{2})(\omega_{1}-\omega_{2})}\\
&\ \ + \frac{x^{n+a+1}j_{1}}{x^{2}+x+1}-\frac{\omega_{1}^{n+a+1}j_{1}}{(x-\omega_{1})(\omega_{1}-\omega_{2})}+\frac{\omega_{2}^{n+a+1}j_{1}}{(x-\omega_{2})(\omega_{1}-\omega_{2})}\\
&\ \ + \frac{x^{n+b+1}j_{2}}{x^{2}+x+1}-\frac{\omega_{1}^{n+b+1}j_{2}}{(x-\omega_{1})(\omega_{1}-\omega_{2})}+\frac{\omega_{2}^{n+b+1}j_{2}}{(x-\omega_{2})(\omega_{1}-\omega_{2})}\\
&\ \ + \frac{x^{n+c+1}j_{3}}{x^{2}+x+1}-\frac{\omega_{1}^{n+c+1}j_{3}}{(x-\omega_{1})(\omega_{1}-\omega_{2})}+\frac{\omega_{2}^{n+c+1}j_{3}}{(x-\omega_{2})(\omega_{1}-\omega_{2})}\\
&=\frac{x^{n+1}\Theta}{x^{2}+x+1}-\frac{\omega_{1}^{n+1}\Phi_{1}}{(x-\omega_{1})(\omega_{1}-\omega_{2})}+\frac{\omega_{2}^{n+1}\Phi_{2}}{(x-\omega_{2})(\omega_{1}-\omega_{2})},
\end{align*}
where $\Theta=1+x^{a}j_{1}+x^{b}j_{2}+x^{c}j_{3}$, $\Phi_{1}=1+\omega_{1}^{a}j_{1}+\omega_{1}^{b}j_{2}+\omega_{1}^{c}j_{3}$ and $\Phi_{2}=1+\omega_{2}^{a}j_{1}+\omega_{2}^{b}j_{2}+\omega_{2}^{c}j_{3}$.
\end{proof}

By Theorem \ref{bb}, we get the Binet formula for the bihyperbolic third-order Jacobsthal polynomials.
\begin{corollary}
Let $n\geq 0$ be an integer. Then, 
\begin{align*}
\mathcal{BJ}_{n}(x)&=\frac{x^{n+1}(1+xj_{1}+x^{2}j_{2}+x^{3}j_{3})}{x^{2}+x+1}\\
&\ \ - \frac{\omega_{1}^{n+1}(1+\omega_{1}j_{1}+\omega_{2}j_{2}+j_{3})}{(x-\omega_{1})(\omega_{1}-\omega_{2})}\\
&\ \ + \frac{\omega_{2}^{n+1}(1+\omega_{2}j_{1}+\omega_{1}j_{2}+j_{3})}{(x-\omega_{2})(\omega_{1}-\omega_{2})}.
\end{align*}
\end{corollary}

For simplicity of notation, let
\begin{equation}\label{sam}
\mathcal{Z}_{n}(x)=\frac{1}{\omega_{1}-\omega_{2}}\left((\omega_{1}x-1)\omega_{1}^{n}\Phi_{1}-(\omega_{2}x-1)\omega_{2}^{n}\Phi_{2}\right).
\end{equation}
Then, the Binet formula of the generalized bihyperbolic third-order Jacobsthal polynomials is given by
\begin{equation}\label{sim}
\mathcal{BJ}_{n}^{(a,b,c)}(x)=\frac{1}{x^{2}+x+1}\left(x^{n+1}\Theta-\mathcal{Z}_{n}(x)\right).
\end{equation}
Note that $\mathcal{Z}_{n}(x)+\mathcal{Z}_{n+1}(x)+\mathcal{Z}_{n+2}(x)=0$.

Assume that $a\geq 1$,  $b\geq 1$, $c\geq 1$ are integers. The Vadja's identity for the sequence $\mathcal{Z}_{n}(x)$ and generalized bihyperbolic third-order Jacobsthal polynomials is given in the next theorem.
\begin{theorem}\label{pp}
Let $n\geq 0$,  $p\geq 0$, $q\geq 0$ be integers. Then, we have
\begin{equation}\label{t1}
\mathcal{Z}_{n+p}(x)\mathcal{Z}_{n+q}(x)-\mathcal{Z}_{n}(x)\mathcal{Z}_{n+p+q}(x)=(x^{2}+x+1)\mathcal{A}_{p}\mathcal{A}_{q}\Phi_{1}\Phi_{2},
\end{equation}
\begin{equation}\label{t2}
\begin{aligned}
\mathcal{BJ}_{n+p}^{(a,b,c)}(x)&\mathcal{BJ}_{n+q}^{(a,b,c)}(x)-\mathcal{BJ}_{n}^{(a,b,c)}(x)\mathcal{BJ}_{n+p+q}^{(a,b,c)}(x)\\
&=\frac{1}{(x^{2}+x+1)^{2}}\left\lbrace \begin{array}{cc} 
(x^{2}+x+1)\mathcal{A}_{p}\mathcal{A}_{q}\Phi_{1}\Phi_{2}\\
+ x^{n+1}\Theta \left(\mathcal{B}_{n+q}(p)-x^{q}\mathcal{B}_{n}(p)\right) 
\end{array} \right\rbrace,
\end{aligned}
\end{equation}
where $\mathcal{A}_{n}=\frac{\omega_{1}^{n}-\omega_{2}^{n}}{\omega_{1}-\omega_{2}}$ and $\mathcal{B}_{n}(p)=\mathcal{Z}_{n+p}(x)-x^{p}\mathcal{Z}_{n}(x)$.
\end{theorem}
\begin{proof}
(\ref{t1}): Using Eq. (\ref{sam}), $A=\omega_{1}x-1$ and $B=\omega_{2}x-1$, we have
\begin{align*}
(\omega_{1}-\omega_{2})^{2}&\left[\mathcal{Z}_{n+p}(x)\mathcal{Z}_{n+q}(x)-\mathcal{Z}_{n}(x)\mathcal{Z}_{n+p+q}(x)\right]\\
&=\left(A\omega_{1}^{n+p}\Phi_{1}-B\omega_{2}^{n+p}\Phi_{2}\right)\left(A\omega_{1}^{n+q}\Phi_{1}-B\omega_{2}^{n+q}\Phi_{2}\right)\\
&\ \ - \left(A\omega_{1}^{n}\Phi_{1}-B\omega_{2}^{n}\Phi_{2}\right)\left(A\omega_{1}^{n+p+q}\Phi_{1}-B\omega_{2}^{n+p+q}\Phi_{2}\right)\\
&=AB(\omega_{1}^{p}-\omega_{2}^{p})\left(\omega_{1}^{q}\Phi_{2}\Phi_{1}-\omega_{2}^{q}\Phi_{1}\Phi_{2}\right)\\
&=(\omega_{1}-\omega_{2})^{2}(x^{2}+x+1)\mathcal{A}_{p}\mathcal{A}_{q}\Phi_{1}\Phi_{2},
\end{align*}
where $\mathcal{A}_{n}=\frac{\omega_{1}^{n}-\omega_{2}^{n}}{\omega_{1}-\omega_{2}}$.

By formula (\ref{sim}) and Eq. (\ref{t1}), we get
\begin{align*}
(x^{2}+x+1)^{2}&\left[\mathcal{BJ}_{n+p}^{(a,b,c)}(x)\mathcal{BJ}_{n+q}^{(a,b,c)}(x)-\mathcal{BJ}_{n}^{(a,b,c)}(x)\mathcal{BJ}_{n+p+q}^{(a,b,c)}(x)\right]\\
&=\left(x^{n+p+1}\Theta-\mathcal{Z}_{n+p}(x)\right)\left(x^{n+q+1}\Theta-\mathcal{Z}_{n+q}(x)\right)\\
&\ \ - \left(x^{n+1}\Theta-\mathcal{Z}_{n}(x)\right)\left(x^{n+p+q+1}\Theta-\mathcal{Z}_{n+p+q}(x)\right)\\
&=\mathcal{Z}_{n+p}(x)\mathcal{Z}_{n+q}(x)-\mathcal{Z}_{n}(x)\mathcal{Z}_{n+p+q}(x)\\
&\ \ + x^{n+1}\Theta \left(\mathcal{B}_{n+q}(p)-x^{q}\mathcal{B}_{n}(p)\right)\\
&=(x^{2}+x+1)\mathcal{A}_{p}\mathcal{A}_{q}\Phi_{1}\Phi_{2}+ x^{n+1}\Theta \left(\mathcal{B}_{n+q}(p)-x^{q}\mathcal{B}_{n}(p)\right),
\end{align*}
where $\mathcal{B}_{n}(p)=\mathcal{Z}_{n+p}(x)-x^{p}\mathcal{Z}_{n}(x)$.
\end{proof}

It is easily seen that for special values of $p$ and $q$ by Theorem \ref{pp}, we get new identities for generalized bihyperbolic third-order Jacobsthal polynomials:
\begin{itemize}
\item Catalan's identity: $q=-p$.
\item Cassini's identity: $p=1$, $q=-1$.
\item d'Ocagne's identity: $p=1$, $q=m-n$, with $m\geq n$.
\end{itemize}

\begin{corollary}
Catalan identity for generalized bihyperbolic third-order Jacobsthal polynomials. Let $n\geq 0$, $p\geq 0$ be integers such that $n\geq p$. Then
\begin{equation}\label{c1}
\begin{aligned}
\mathcal{BJ}_{n+p}^{(a,b,c)}(x)&\mathcal{BJ}_{n-p}^{(a,b,c)}(x)-\left(\mathcal{BJ}_{n}^{(a,b,c)}(x)\right)^{2}\\
&=\frac{1}{(x^{2}+x+1)^{2}}\left\lbrace \begin{array}{cc} 
-(x^{2}+x+1)\mathcal{A}_{p}^{2}\Phi_{1}\Phi_{2}\\
+ x^{n+1}\Theta \left(\mathcal{B}_{n-p}(p)-x^{-p}\mathcal{B}_{n}(p)\right) 
\end{array} \right\rbrace.
\end{aligned}
\end{equation}
\end{corollary}

\begin{corollary}
Cassini identity for generalized bihyperbolic third-order Jacobsthal polynomials. Let $n\geq 1$ be an integer. Then
\begin{equation}\label{c2}
\begin{aligned}
\mathcal{BJ}_{n+1}^{(a,b,c)}(x)&\mathcal{BJ}_{n-1}^{(a,b,c)}(x)-\left(\mathcal{BJ}_{n}^{(a,b,c)}(x)\right)^{2}\\
&=\frac{1}{(x^{2}+x+1)^{2}}\left\lbrace \begin{array}{cc} 
-(x^{2}+x+1)\Phi_{1}\Phi_{2}\\
+ x^{n+1}\Theta \left(\mathcal{B}_{n-1}(1)-x^{-1}\mathcal{B}_{n}(1)\right) 
\end{array} \right\rbrace.
\end{aligned}
\end{equation}
\end{corollary}

\begin{corollary}
d'Ocagne identity for generalized bihyperbolic third-order Jacobsthal polynomials. Let $n\geq 0$, $m\geq 0$ be integers such that $n\geq m$. Then
\begin{equation}\label{t2}
\begin{aligned}
\mathcal{BJ}_{n+1}^{(a,b,c)}(x)&\mathcal{BJ}_{m}^{(a,b,c)}(x)-\mathcal{BJ}_{n}^{(a,b,c)}(x)\mathcal{BJ}_{m+1}^{(a,b,c)}(x)\\
&=\frac{1}{(x^{2}+x+1)^{2}}\left\lbrace \begin{array}{cc} 
(x^{2}+x+1)\mathcal{A}_{m-n}\Phi_{1}\Phi_{2}\\
+ x^{n+1}\Theta \left(\mathcal{B}_{m}(1)-x^{m-n}\mathcal{B}_{n}(1)\right) 
\end{array} \right\rbrace.
\end{aligned}
\end{equation}
\end{corollary}

Now, we give the ordinary generating function for the generalized bihyperbolic third-order Jacobsthal polynomials. 
\begin{theorem}
The generating function for the generalized bihyperbolic third-order Jacobsthal polynomial sequence $\mathcal{BJ}_{n}^{(a,b,c)}(x)$ is 
$$g_{\mathcal{BJ}}(t;x)=\frac{\left\lbrace \begin{array}{cc} 
\mathcal{BJ}_{0}^{(a,b,c)}(x)+\left(\mathcal{BJ}_{1}^{(a,b,c)}(x)-(x-1)\mathcal{BJ}_{0}^{(a,b,c)}(x)\right)t\\
+ \left(\mathcal{BJ}_{2}^{(a,b,c)}(x)-(x-1)\mathcal{BJ}_{1}^{(a,b,c)}(x)-(x-1)\mathcal{BJ}_{0}^{(a,b,c)}(x)\right)t^{2}
\end{array} \right\rbrace}{1-(x-1)t-(x-1)t^{3}-xt^{3}}.$$
\end{theorem}
\begin{proof}
Let $$g_{\mathcal{BJ}}(t;x)=\mathcal{BJ}_{0}^{(a,b,c)}(x)+\mathcal{BJ}_{1}^{(a,b,c)}(x)t+\mathcal{BJ}_{2}^{(a,b,c)}(x)t^{2}+\cdots + \mathcal{BJ}_{n}^{(a,b,c)}(x)t^{n}+\cdots $$
be the generating function of the generalized bihyperbolic third-order Jacobsthal polynomials. Hence we have
\begin{align*}
(x-1)tg_{\mathcal{BJ}}(t;x)&=(x-1)\mathcal{BJ}_{0}^{(a,b,c)}(x)t+(x-1)\mathcal{BJ}_{1}^{(a,b,c)}(x)t^{2}\\
&\ \ + (x-1)\mathcal{BJ}_{2}^{(a,b,c)}(x)t^{3}+\cdots \\
(x-1)t^{2}g_{\mathcal{BJ}}(t;x)&=(x-1)\mathcal{BJ}_{0}^{(a,b,c)}(x)t^{2}+(x-1)\mathcal{BJ}_{1}^{(a,b,c)}(x)t^{3}\\
&\ \ + (x-1)\mathcal{BJ}_{2}^{(a,b,c)}(x)t^{4}+\cdots \\
xt^{3}g_{\mathcal{BJ}}(t;x)&=x\mathcal{BJ}_{0}^{(a,b,c)}(x)t^{3}+x\mathcal{BJ}_{1}^{(a,b,c)}(x)t^{4}+x\mathcal{BJ}_{2}^{(a,b,c)}(x)t^{5}+\cdots 
\end{align*}
Using the recurrence (\ref{rec}), we get 
\begin{align*}
g_{\mathcal{BJ}}(t;x)&\left(1-(x-1)t-(x-1)t^{3}-xt^{3}\right)\\
&=\left\lbrace \begin{array}{cc} 
\mathcal{BJ}_{0}^{(a,b,c)}(x)+\left(\mathcal{BJ}_{1}^{(a,b,c)}(x)-(x-1)\mathcal{BJ}_{0}^{(a,b,c)}(x)\right)t\\
+ \left(\mathcal{BJ}_{2}^{(a,b,c)}(x)-(x-1)\mathcal{BJ}_{1}^{(a,b,c)}(x)-(x-1)\mathcal{BJ}_{0}^{(a,b,c)}(x)\right)t^{2}
\end{array} \right\rbrace.
\end{align*}
Thus
$$g_{\mathcal{BJ}}(t;x)=\frac{\left\lbrace \begin{array}{cc} 
\mathcal{BJ}_{0}^{(a,b,c)}(x)+\left(\mathcal{BJ}_{1}^{(a,b,c)}(x)-(x-1)\mathcal{BJ}_{0}^{(a,b,c)}(x)\right)t\\
+ \left(\mathcal{BJ}_{2}^{(a,b,c)}(x)-(x-1)\mathcal{BJ}_{1}^{(a,b,c)}(x)-(x-1)\mathcal{BJ}_{0}^{(a,b,c)}(x)\right)t^{2}
\end{array} \right\rbrace}{1-(x-1)t-(x-1)t^{3}-xt^{3}}.$$
\end{proof}

In the next theorem we give a summation formula for the generalized bihyperbolic third-order Jacobsthal polynomials. In the proof we will use the following result.
\begin{lemma}[Proposition 5, \cite{Ce}]
If $\mathcal{J}_{n}^{(3)}(x)$ is the $n$-th term of the third-order Jacobsthal polynomial sequence, then
\begin{equation}\label{sum}
\sum_{s=0}^{n}\mathcal{J}_{n}^{(3)}(x)=\frac{1}{3(x-1)}\left(\mathcal{J}_{n+2}^{(3)}(x)-(x-2)\mathcal{J}_{n+1}^{(3)}(x)+x\mathcal{J}_{n}^{(3)}(x)-1\right).
\end{equation}
\end{lemma}

\begin{theorem}
Let $n\geq 0$, $a\geq 1$, $b\geq 1$, $c\geq 1$, be integers. Then
\begin{equation}\label{suma}
\sum_{s=0}^{n}\mathcal{BJ}_{n}^{(a,b,c)}(x)=\frac{1}{3(x-1)}\left\lbrace \begin{array}{cc} 
(2x-3)\mathcal{BJ}_{0}^{(a,b,c)}(x)+(x-2)\mathcal{BJ}_{1}^{(a,b,c)}(x)\\
- \mathcal{BJ}_{2}^{(a,b,c)}(x)+\mathcal{BJ}_{n+2}^{(a,b,c)}(x)\\
- (x-2)\mathcal{BJ}_{n+1}^{(a,b,c)}(x)+x\mathcal{BJ}_{n}^{(a,b,c)}(x)
\end{array} \right\rbrace.
\end{equation}
\end{theorem}
\begin{proof}
We use induction on $n$. If $n=0$, we obtain
\begin{align*}
\mathcal{BJ}_{0}^{(a,b,c)}(x)&=\frac{1}{3(x-1)}\left\lbrace \begin{array}{cc} 
(2x-3)\mathcal{BJ}_{0}^{(a,b,c)}(x)+(x-2)\mathcal{BJ}_{1}^{(a,b,c)}(x)\\
- \mathcal{BJ}_{2}^{(a,b,c)}(x)+\mathcal{BJ}_{2}^{(a,b,c)}(x)\\
- (x-2)\mathcal{BJ}_{1}^{(a,b,c)}(x)+x\mathcal{BJ}_{0}^{(a,b,c)}(x)
\end{array} \right\rbrace \\
&=3(x-1)\mathcal{BJ}_{0}^{(a,b,c)}(x),
\end{align*}
then the result is obvious. Assuming the formula (\ref{suma}) holds for $n\geq 0$, we shall prove it for $n+1$. Using the induction's hypothesis and formula (\ref{rec}), we have
\begin{align*}
\sum_{s=0}^{n+1}\mathcal{BJ}_{n}^{(a,b,c)}(x)&=\sum_{s=0}^{n}\mathcal{BJ}_{n}^{(a,b,c)}(x)+\mathcal{BJ}_{n+1}^{(a,b,c)}(x)\\
&=\frac{1}{3(x-1)}\left\lbrace \begin{array}{cc} 
(2x-3)\mathcal{BJ}_{0}^{(a,b,c)}(x)+(x-2)\mathcal{BJ}_{1}^{(a,b,c)}(x)\\
- \mathcal{BJ}_{2}^{(a,b,c)}(x)+\mathcal{BJ}_{n+2}^{(a,b,c)}(x)\\
- (x-2)\mathcal{BJ}_{n+1}^{(a,b,c)}(x)+x\mathcal{BJ}_{n}^{(a,b,c)}(x)\\
+ (3x-3)\mathcal{BJ}_{n+1}^{(a,b,c)}(x)
\end{array} \right\rbrace \\
&=\frac{1}{3(x-1)}\left\lbrace \begin{array}{cc} 
(2x-3)\mathcal{BJ}_{0}^{(a,b,c)}(x)+(x-2)\mathcal{BJ}_{1}^{(a,b,c)}(x)\\
- \mathcal{BJ}_{2}^{(a,b,c)}(x)+\mathcal{BJ}_{n+3}^{(a,b,c)}(x)\\
- (x-2)\mathcal{BJ}_{n+2}^{(a,b,c)}(x)+x\mathcal{BJ}_{n+1}^{(a,b,c)}(x)
\end{array} \right\rbrace,
\end{align*}
which ends the proof.
\end{proof}

At the end, we give matrix representations of the polynomials $\mathcal{BJ}_{n}^{(a,b,c)}(x)$. By the equality (\ref{rec}) we get the following result.
\begin{theorem}
Let $n\geq 1$ be an integer. Then
\begin{equation}
\left[
\begin{array}{c}
\mathcal{BJ}_{n+2}^{(a,b,c)}(x) \\ 
\mathcal{BJ}_{n+1}^{(a,b,c)}(x) \\ 
\mathcal{BJ}_{n}^{(a,b,c)}(x)
\end{array}
\right]=\textbf{Q}_{\mathcal{J}} \cdot\left[
\begin{array}{c}
\mathcal{BJ}_{n+1}^{(a,b,c)}(x) \\ 
\mathcal{BJ}_{n}^{(a,b,c)}(x) \\ 
\mathcal{BJ}_{n-1}^{(a,b,c)}(x)
\end{array}
\right],
\end{equation}
where $$\textbf{Q}_{\mathcal{J}}=\left[
\begin{array}{ccc}
x-1& x-1& x \\ 
1& 0 & 0 \\ 
0 & 1& 0
\end{array}
\right].$$
\end{theorem}

\begin{theorem}
Let $n\geq 0$ be an integer. Then
\begin{equation}\label{mat}
\left[
\begin{array}{ccc}
\mathcal{BJ}_{n+3}^{(a,b,c)}& \mathcal{T}_{n+4}^{(a,b,c)}& x\mathcal{BJ}_{n+2}^{(a,b,c)} \\ 
\mathcal{BJ}_{n+2}^{(a,b,c)}& \mathcal{T}_{n+3}^{(a,b,c)}& x\mathcal{BJ}_{n+1}^{(a,b,c)} \\
\mathcal{BJ}_{n+1}^{(a,b,c)}& \mathcal{T}_{n+2}^{(a,b,c)}& x\mathcal{BJ}_{n}^{(a,b,c)}
\end{array}
\right]=\left[
\begin{array}{ccc}
\mathcal{BJ}_{3}^{(a,b,c)}& \mathcal{T}_{4}^{(a,b,c)}& x\mathcal{BJ}_{2}^{(a,b,c)} \\ 
\mathcal{BJ}_{2}^{(a,b,c)}& \mathcal{T}_{3}^{(a,b,c)}& x\mathcal{BJ}_{1}^{(a,b,c)} \\
\mathcal{BJ}_{1}^{(a,b,c)}& \mathcal{T}_{2}^{(a,b,c)}& x\mathcal{BJ}_{0}^{(a,b,c)}
\end{array}
\right]\cdot \textbf{Q}_{\mathcal{J}}^{n},
\end{equation}
where $\mathcal{T}_{n}^{(a,b,c)}=\mathcal{BJ}_{n}^{(a,b,c)}-(x-1)\mathcal{BJ}_{n-1}^{(a,b,c)}$.
\end{theorem}
\begin{proof}
We use induction on $n$. If $n=0$ then the result is obvious. Assuming the formula (\ref{mat}) holds for $n\geq 0$, we shall prove it for $n+1$. Using the induction's hypothesis and formula (\ref{rec}), we have
\begin{align*}
&\left[
\begin{array}{ccc}
\mathcal{BJ}_{3}^{(a,b,c)}& \mathcal{T}_{4}^{(a,b,c)}& x\mathcal{BJ}_{2}^{(a,b,c)} \\ 
\mathcal{BJ}_{2}^{(a,b,c)}& \mathcal{T}_{3}^{(a,b,c)}& x\mathcal{BJ}_{1}^{(a,b,c)} \\
\mathcal{BJ}_{1}^{(a,b,c)}& \mathcal{T}_{2}^{(a,b,c)}& x\mathcal{BJ}_{0}^{(a,b,c)}
\end{array}
\right]\cdot \textbf{Q}_{\mathcal{J}}^{n+1}\\
&=\left[
\begin{array}{ccc}
\mathcal{BJ}_{3}^{(a,b,c)}& \mathcal{T}_{4}^{(a,b,c)}& x\mathcal{BJ}_{2}^{(a,b,c)} \\ 
\mathcal{BJ}_{2}^{(a,b,c)}& \mathcal{T}_{3}^{(a,b,c)}& x\mathcal{BJ}_{1}^{(a,b,c)} \\
\mathcal{BJ}_{1}^{(a,b,c)}& \mathcal{T}_{2}^{(a,b,c)}& x\mathcal{BJ}_{0}^{(a,b,c)}
\end{array}
\right]\cdot \textbf{Q}_{\mathcal{J}}^{n}\cdot \left[
\begin{array}{ccc}
x-1& x-1& x \\ 
1& 0 & 0 \\ 
0 & 1& 0
\end{array}
\right]\\
&=\left[
\begin{array}{ccc}
\mathcal{BJ}_{n+3}^{(a,b,c)}& \mathcal{T}_{n+4}^{(a,b,c)}& x\mathcal{BJ}_{n+2}^{(a,b,c)} \\ 
\mathcal{BJ}_{n+2}^{(a,b,c)}& \mathcal{T}_{n+3}^{(a,b,c)}& x\mathcal{BJ}_{n+1}^{(a,b,c)} \\
\mathcal{BJ}_{n+1}^{(a,b,c)}& \mathcal{T}_{n+2}^{(a,b,c)}& x\mathcal{BJ}_{n}^{(a,b,c)}
\end{array}
\right]\left[
\begin{array}{ccc}
x-1& x-1& x \\ 
1& 0 & 0 \\ 
0 & 1& 0
\end{array}
\right]\\
&=\left[
\begin{array}{ccc}
\mathcal{BJ}_{n+4}^{(a,b,c)}&(x-1)\mathcal{BJ}_{n+3}^{(a,b,c)}+x\mathcal{BJ}_{n+2}^{(a,b,c)}& x\mathcal{BJ}_{n+2}^{(a,b,c)} \\ 
\mathcal{BJ}_{n+3}^{(a,b,c)}&(x-1)\mathcal{BJ}_{n+2}^{(a,b,c)}+x\mathcal{BJ}_{n+1}^{(a,b,c)}& x\mathcal{BJ}_{n+1}^{(a,b,c)} \\
\mathcal{BJ}_{n+2}^{(a,b,c)}&(x-1)\mathcal{BJ}_{n+1}^{(a,b,c)}+x\mathcal{BJ}_{n}^{(a,b,c)}& x\mathcal{BJ}_{n}^{(a,b,c)}
\end{array}
\right]\\
&=\left[
\begin{array}{ccc}
\mathcal{BJ}_{n+4}^{(a,b,c)}& \mathcal{T}_{n+5}^{(a,b,c)}& x\mathcal{BJ}_{n+3}^{(a,b,c)} \\ 
\mathcal{BJ}_{n+3}^{(a,b,c)}& \mathcal{T}_{n+4}^{(a,b,c)}& x\mathcal{BJ}_{n+2}^{(a,b,c)} \\
\mathcal{BJ}_{n+2}^{(a,b,c)}& \mathcal{T}_{n+3}^{(a,b,c)}& x\mathcal{BJ}_{n+1}^{(a,b,c)}
\end{array}
\right]
\end{align*}
which ends the proof.
\end{proof}

\section{Conclusions}
In this study, we defined the generalized bihyperbolic third-order Jacobsthal polynomials, which is an extension of the third-order Jacobsthal polynomials. For the generalized bihyperbolic third-order Jacobsthal polynomials, we provided a range of identities, including the Catalan and Cassini identity. For future research, additional identities and generalizations of the third-order modified Jacobsthal polynomials can be studied.


\medskip

\end{document}